\documentclass[11pt, notitlepage]{article}
\usepackage{amssymb,amsmath,comment}
\catcode`\@=11 \@addtoreset{equation}{section}
\def\thesection{\arabic{section}}

\def\theequation{\thesection.\arabic{equation}}
\catcode`\@=12
\usepackage{colortbl}
\usepackage{epsf}
\usepackage{esint}
\usepackage{a4wide}
\usepackage[mathscr]{eucal}
\usepackage{epsf}
\usepackage{a4wide}

\newcommand{\noi} {\noindent}

\usepackage[all]{xy}
\catcode`\@=11
\def\theequation{\@arabic{\c@section}.\@arabic{\c@equation}}
\catcode`\@=12

\newtheorem{Theorem}{Theorem}[section]
\newtheorem{Lemma}[Theorem]{Lemma}

\newtheorem{Definition}[Theorem]{Definition}

\begin{document}

{\vspace{0.01in}}
\title{Nonexistence of variational minimizers related to a quasilinear singular problem in metric measure spaces}

\author{Prashanta Garain and Juha Kinnunen}

\maketitle

\begin{abstract}
In this article we consider a variational problem related to a quasilinear singular problem and obtain a nonexistence result in a metric measure space with a doubling measure and a Poincar\'e inequality. Our method is purely variational and to the best of our knowledge, this is the first work concerning singular problems in a general metric setting.
\medskip

\noi \textit{Key words:} Analysis on metric measure spaces, Newtonian space, variational integral, minimizer, energy estimate, nonexistence.

\medskip

\noi \textit{2020 Mathematics Subject Classification:} 35A01, 35A15, 35B45, 35J92.

\end{abstract}

\section{Introduction}
In this article, we discuss the question of nonexistence of variational minimizers related to a quasilinear singular problem
\begin{equation}\label{Eqn Motive to define stable}
-\Delta_p u=
-\operatorname{div}(|\nabla u|^{p-2}\nabla u)=f(u).
\end{equation} 
Singular problem refers to a blow up near zero of the nonlinearity on the right-hand side of \eqref{Eqn Motive to define stable}. 
Singular elliptic problems has been a topic of considerable attention over the last three decades and there is a colossal amount of literature available concerning the question of existence, uniqueness, nonexistence and regularity. Most of these results are in the Euclidean setting, but recently related result have been obtained on Riemannian manifolds as well.

Stable solution for the equation \eqref{Eqn Motive to define stable} in the Euclidean case is defined to be a function $u$ in a appropriate Sobolev space satisfying 
\begin{equation*}
\int_{\mathbb{R}^N}|\nabla u|^{p-2}\nabla u\cdot\nabla\phi\,dx=\int_{\mathbb{R}^N}f(u)\phi\,dx
\end{equation*}
and
\begin{equation*}
\int_{\mathbb{R}^N}f'(u)\phi^2\,dx\leq\int_{\mathbb{R}^N}\big(|\nabla u|^{p-2}|\nabla\phi|^2+(p-2)|\nabla u|^{p-4}|\langle\nabla u,\nabla\phi\rangle|^2\big)\,dx
\end{equation*}
simultaneously for every $\phi\in C_c^{\infty}(\mathbb{R}^N)$. 
On a metric measure space, we consider a notion of variational minimizer (Definition \ref{Variationalmin}) analogous to stable solution where the role of modulus of the gradient in the Euclidean space is played by the upper gradient. The main goal in this article is to provide a nonexistence result concerning stable solutions to the problem
\begin{equation}\label{plap}
\Delta_p u=u^{-\delta}
\end{equation}
in $X$, with $p\ge2$, where $X$ is a metric measure space with a doubling measure and a Poincar\'e inequality (see Section 2 for the definitions). 
We establish an energy estimate (Lemma \ref{energy}) that is an important ingredient to prove our main result (Theorem \ref{mainthm}).
Our main result is valid for stable solutions in Newtonian spaces, a counterpart of Sobolev spaces in the Euclidean case, 
up to a dimension of the metric space related to the underlying measure (see Section 2).
To the best of our knowledge, till date there is no literature available concerning nonexistence of stable solutions to singular problem in a general metric measure space setting. 

To motivate our present study, let us give an overview of the available results related to our question. In the Euclidean space $\mathbb{R}^N$, Crandall et al. proved in \cite{CRT} that, for any bounded smooth domain $\Omega$, the Dirichlet problem related to the singular problem $-\Delta u=u^{-\delta}$ admits a unique positive classical solution for any $\delta>0$.
Lazer and McKenna proved in \cite{LcMc} that such solutions are weak solutions for $0<\delta<3$. 
This restriction on $\delta$ was removed by Boccardo and Orsina in \cite{BocOrs} and has further been extended by De Cave \cite{Decave} and Canino et al. \cite{Canino} to the $p$-Laplace equation $-\Delta_p u=u^{-\delta}$ for any $\delta>0$. 
The existence and nonexistence of ground state solutions to the problem of type $-\Delta_p u=f(x)u^{-\delta}$ in $\mathbb R^N$ are also investigated under suitable assumptions on $f$, see C$\hat{\text{i}}$rstea and R$\breve{\text{a}}$dulescu \cite{Radhupaper}, Ghergu and R$\breve{\text{a}}$dulescu \cite{Radhubook} and the references therein. Moreover, stable solutions for equations of the type $-\Delta_p u=f(u)$ have been investigated by several authors for various types of nonlinearity $f$. In the semilinear case $p=2$, Dupaigne and Farina in \cite{DupFar} proved that for a nonnegative $f\in C^1(\mathbb{R})$, any bounded stable solution $u\in C^2(\mathbb{R}^N)$ of the problem $-\Delta u=f(u)$ is constant for $1\leq N\leq 4$. 

Very recently a long standing conjecture related to the optimal regularity of stable solutions for the equation $-\Delta_p u=f(u)$ was solved by Cabr$\acute{\text{e}}$ et al. in \cite{CabMirSan} under the assumptions that $f\in C^1(\Omega)$, $\Omega$ is a strictly convex bounded smooth domain in $\mathbb{R}^N$, $N<p+\frac{4p}{p-1}$ and $p>2$.
The semilinear case $p=2$ up to dimension $9$ has been studied by Cabr$\acute{\text{e}}$ et al. in \cite{Cabre}. 
We also refer the reader to \cite{Cabre3, CabCap, Cabre-RosOton, DupFar2, Dupaignebook, Farina1, Nedev} for nonsingular $f$. Castorina et al. in \cite{Castorina} studied nonexistence results to the $p$-Laplace equations involving the polynomial and power type nonlinearites, e,g. $f(u)=e^u,(1+u)^{-\frac{1}{\delta-1}}$ when $\delta>1$ and $(1+u)^\frac{1}{1-\delta}$ when $0<\delta<1$. Almost simultaneously for the singular case, $f(u)=-u^{-\delta}$, Ma and Wei in \cite{MaWei} proved that the problem $\Delta u=u^{-\delta}$ does not admit any positive $C^2(\mathbb{R}^N)$ stable solution for any $\delta>0$, provided 
$$2\leq N<2+\frac{4}{1+\delta}\Big(\delta+\sqrt{\delta^2+\delta}\Big).$$ 
Similar results are obtained for the $p$-Laplace equation $-\Delta_p u=f(u)$ in $\mathbb{R}^N$ with singular nonlinearity $f$ in lower dimension, see \cite{Chen,Ple2} and the references therein.

In the Riemannian setting, Farina et al. \cite{FarinaMariEnrico} studied qualitative properties of stable solutions to the semilinear problem $-\Delta u=f(u)$ for nonsingular $f$ to determine the structure of the manifold. For the Lane-Emden-Fowler nonlinearity $f(s)=|s|^{p-1}s$, Berchio et al. in \cite{Ber} studied existence, uniqueness and stability of radial solutions in a Riemannian manifold. Recently in \cite{Marcos}, do $\acute{\text{O}}$ and Clemente obtained some regularity properties of semistable solutions for the Gelfand nonlinearity $f(u)=e^u$, power nonlinearity $(1+u)^{m}$ for $m>1$ and they also considered the singular case $(1-s)^{-2}$ in a bounded smooth domain of a complete Riemannian manifold. 

This article is organized as follows. In Section 2, we present some basic definitions, preliminaries for our set up and state our main result. In Section 3, we focus on proving the energy estimate and in Section 4, our main result is proved.

\section{Preliminaries and main result}
This section is devoted to some basic definitions and known results in metric measure spaces and the statement of the main result. For a more detailed discussion on the theory of metric measure spaces, see Bj$\ddot{\text{o}}$rn \cite{AJBjorn}, Heinonen \cite{Heinonen1}, Heinonen et al. \cite{Heinonen2} and the references therein.

Throughout the article, we assume that the triplet $(X,d,\mu)$ is a metric measure space, where $d$ is the metric and $\mu$ is a Borel measure on $X$. 
We define for $x\in X$, the ball with center at $x$ and of radius $r>0$ by $B(x,r)=\{y\in X:d(y,x)<r\}$. 
We write $C$ to denote a positive constant which may vary from line to line or even in the same line depending on the situation. Moreover, if $C$ depends on $r_1,r_2,\dots,r_n$, we write $C=C(r_1,r_2,\dots,r_n)$.

Next we define the notion of doubling measure and dimension of the metric space $X$.
\begin{Definition}(Doubling measure)
We say that $\mu$ is a doubling measure, if for every $r>0$ and $x\in X$, there exists a constant $C_\mu\geq 1$, called the doubling constant of $\mu$, such that 
\begin{equation}\label{doubling}
0<\mu(B(x,2r))\leq C_\mu\mu(B(x,r))<\infty.
\end{equation}  
\end{Definition}

The following lemma gives a counterpart of the dimension for a doubling measure in a metric measure space. 
\begin{Lemma}\label{doublinglemma}(Lemma 3.3, \cite{AJBjorn})
Assume $\mu$ is doubling with the doubling constant $C_\mu\geq 1$. 
Then
\begin{equation}\label{doublingimply}
\frac{\mu(B(x,R))}{\mu(B(x,r))}\leq C\Big(\frac{R}{r}\Big)^m,
\end{equation}
for every $x\in X$ and every $0<r\leq R$, with $C=C_\mu^{2}$ and $m=\log_2\,C_\mu$.
\end{Lemma}


We define the notion of upper gradient of an extended real valued function in metric measure space by path integrals.
A path is a continuous mapping from a compact  one-dimensional interval to $X$.
A path is rectifiable, if it has finite length. We refer to \cite{AJBjorn, Heinonen1, Heinonen2} for more on paths, path integrals and $p$-modulus of a path family.

\begin{Definition}\label{uppergrad}(Upper gradient)
A Borel function $g:X\to[0,\infty]$ is an upper gradient of a function $u:X\to[-\infty,\infty]$, if for every rectifiable path $\gamma$ joining $x$ and $y$, we have
\begin{equation}\label{uppergraddef}
|u(x)-u(y)|\leq \int_{\gamma}g\,ds.
\end{equation}
If $u(x)=u(y)=\infty$ or $u(x)=u(y)=-\infty$, then we define the left side of \eqref{uppergraddef} to be $\infty$. 
We say that a nonnegative measurable function $g$ is a $p$-weak upper gradient of $u$, 
if \eqref{uppergraddef} holds outside a path family of zero $p$-modulus.   
\end{Definition}
If $u$ is Lipschitz, then the Lipschitz constant is an upper gradient of $u$, see \cite{AJBjorn}. 
Now we define the notion of Sobolev space on a metric measure space, see \cite{AJBjorn, Heinonen2, Nagametric}.
\begin{Definition}\label{Newton}(Newtonian space)
For $1\leq p<\infty$, we define $\bar{N}^{1,p}(X)$ to be the space of all $p$-integrable functions $u$ such that $u$ admits a $p$-integrable upper gradient $g$ under the seminorm 
$$
||u||_{\bar{N}^{1,p}(X)}=||u||_{L^p(X)}+\inf||g||_{L^p(X)},
$$ 
the infimum being taken over all upper gradients $g$ of $u$. 
The Newtonian space $N^{1,p}(X)$ is the quotient space $\bar{N}^{1,p}(X)/\sim$ under the norm
$$
||u||_{N^{1,p}(X)}=||u||_{\bar{N}^{1,p}(X)},
$$
where $u\sim v$ if $||u-v||_{\bar{N}^{1,p}(X)}=0$.
The corresponding local space $N^{1,p}_{\mathrm{loc}}(X)$ is defined by requiring that every point $x\in X$ has an open neighborhood $U_x$ in $X$ such that $u\in N^{1,p}(U_x)$. 
\end{Definition}

A result in \cite{Nagametric} asserts that $N^{1,p}(X)$ is a Banach space, see also \cite{AJBjorn, Heinonen2}.
If $u$ has an upper gradient $g\in L^p(X)$, then it follows from \cite{Nagametric} that $u$ has a unique minimal upper gradient $g_u$ in the sense that for every $p$-upper gradient $g\in L^p(X)$ of $u$, we have $g_u\leq g$ $\mu$-almost everywhere and
$$
||g_u||_{L^p(X)}=\inf||g||_{L^p(X)},
$$
see also \cite{AJBjorn, Heinonen2}.
For more discussion on upper gradient and further properties of Newtonian spaces, we refer the reader to \cite{AJBjorn, Heinonen1, Heinonen2}.

\begin{Definition}\label{Poincareineq}(Weak $(1,p)$-Poincar\'e inequality)
We say that $X$ supports a weak $(1,p)$-Poincar\'e inequality if there exist a constant $C>0$ and $\lambda\geq 1$ such that
for every $u\in L^1_{\mathrm{loc}}(X)$,  every upper gradient $g$ of $u$ and every ball $B(x,r)$ in $X$, we have
\begin{equation}\label{Poncare}
\fint_{B(x,r)}|u-u_{B(x,r)}|\,d\mu\leq C r\Big(\fint_{B(x,\lambda r)}g^p\,d\mu\Big)^\frac{1}{p},
\end{equation}
where
$$
u_{B(x,r)}
=\fint_{B(x,r)}u\,d\mu
=\frac{1}{\mu(B(x,r))}\int_{B(x,r)}u\,d\mu.
$$  
\end{Definition}

For the rest of the article, we assume that $X$ is a complete metric space with a doubling measure \eqref{doubling} and the weak $(1,p)$-Poincar\'e inequality \eqref{Poncare}. Moreover throughout we assume that $p\geq 2$ unless otherwise mentioned. A complete metric measure space with a doubling measure is proper, that is, closed and bounded sets are compact, see \cite[Proposition 3.1]{AJBjorn}.

Next we define variational minimizers corresponding to the general singular problem \eqref{Eqn Motive to define stable}, although we are mainly interested in the special case $f(u)=-u^{-\delta}$ with $\delta>0$, when we have \eqref{plap}.

\begin{Definition}\label{Variationalmin}(Variational minimizer)
Let $p\geq 2$. We say that a function $u\in N^{1,p}_{\mathrm{loc}}(X)$ which is positive in $X$ is a variational minimizer corresponding to the problem \eqref{Eqn Motive to define stable}, if for every $\phi\in N^{1,p}(X)$ with compact support in $X$ we have $f(u)\phi\in L^1_{\mathrm{loc}}(X)$, $f'(u)\phi^2\in L^1_{\mathrm{loc}}(X)$ and the following two inequalities
\begin{equation}\label{weak}
\int_{X}g_u^{p}\,d\mu\leq\int_{X}g_{u+\phi}^{p}\,d\mu-\int_{X}f(u)\phi\,d\mu,
\end{equation}
and
\begin{equation}\label{stable}
\int_{X}f'(u)\phi^2\,d\mu\leq(p-1)\int_{X}g_u^{p-2}g_{\phi}^2\,d\mu
\end{equation}
hold simultaneously.
\end{Definition}

Before proceeding to state our main result, let us fix notation which will be used throughout.
For $p\geq 2$ and $\delta>0$, we denote
$$
I=\Big(\frac{(p^2-3p+2)^2}{4(p-1)^2},\infty\Big),
\quad
\delta_p=\frac{2\delta+p-1+2\sqrt{\delta(\delta+p-1)}}{p-1}
\quad\text{and}\quad 
J=(p-1,\delta_p).
$$
Observe that $J\neq\emptyset$, if $\delta\in I$.

Next we state our main result concerning nonexistence of variational minimizers.
\begin{Theorem}\label{mainthm}
Assume that $p\geq 2$, $(X,d,\mu)$ is a complete metric space satisfying the doubling property \eqref{doubling}, with the doubling constant $C_\mu$, and $X$ supports the weak $(1,p)$-Poincar\'e inequality \eqref{Poncare}. Let $m=\log_{2}C_\mu$ and $\delta\in I$. 
If
$$
m<\frac{p(\delta+\delta_p)}{\delta+p-1},
$$
there is no positive variational minimizer in $N^{1,p}_{\mathrm{loc}}(X)$ corresponding to the problem \eqref{plap}.
\end{Theorem}
To prove Theorem \ref{mainthm} we establish an energy estimate (Lemma \ref{energy}) for variational minimizers in the next section.

\section{Energy estimate}
This section is devoted to establish energy estimate for variational minimizers which relies on choosing suitable test functions. In this concern let us denote by $\text{Lip}_c(X)$ to be the class of Lipschitz functions with compact support in $X$. For any $\beta>p-1$ and $l\in\mathbb{N}$, we define the truncated functions $F_l$ and $G_l$ by
$$
F_l(t)=
\begin{cases}
-\beta l^{\beta+1}\Big(t-\frac{\beta+1}{l\beta}\Big),\text{ for }0\leq t\leq\frac{1}{l},\\
t^{-\beta},\text{ for }t\geq \frac{1}{l},\\
\end{cases}
$$
and
$$
G_l(t)=
\begin{cases}
\frac{1-\beta}{2}l^\frac{\beta+1}{2}\Big(t+\frac{\beta+1}{l(1-\beta)}\Big),\text{ for }0\leq t\leq\frac{1}{l},\\
t^\frac{1-\beta}{2},\text{ for }t\geq \frac{1}{l}.\\
\end{cases}
$$
It can be easily seen that $F_l$ and $G_l$ are monotone decreasing $C^1([0,\infty))$ functions with the properties
\begin{equation}\label{cond1}
\begin{split}
 G_l(t)^{2}\geq t F_l(t),
\end{split}
\end{equation}
 \begin{equation}\label{cond2}
 \begin{split}
 G_l'(t)^{2}=\frac{(\beta-1)^2}{4\beta}|F_l'(t)|,
 \end{split}
 \end{equation}
 \begin{equation}\label{cond3}
 \begin{split}
 F_l(t)|F_l'(t)|^{1-p}+G_l(t)^{p}|G_l'(t)|^{2-p}&\leq C|t|^{p-\beta-1},\quad\text{for every}\quad t\geq 0,
 \end{split}
 \end{equation}
for some positive constant $C=C(p,\beta)$. We refer to \cite{Marmas1} for a similar technique.

\begin{Lemma}\label{energy}
Let $\delta\in I$, $p\geq 2$ and assume that $u\in N^{1,p}_{\mathrm{loc}}(X)$ is a positive variational minimizer corresponding to the problem \eqref{plap}. Then for every $\beta\in J$ and every $\psi\in\mathrm{Lip}_c(X)$ with $0\leq\psi\leq 1$, there exists a constant $C=C(\delta,p,\beta)>0$ such that
\begin{equation}\label{energyineq}
\int_{X}(u^{-\delta-\beta}+g_u^{p}u^{-\beta-1})\psi^p\,d\mu\leq C\int_{X}u^{p-\beta-1}g_{\psi}^p\,d\mu.
\end{equation}
\end{Lemma}
\begin{proof}
Let $u\in N^{1,p}_{\mathrm{loc}}(X)$ be a positive variational minimizer corresponding to the problem \eqref{plap} and $\psi\in\text{Lip}_{c}(X)$ with $0\leq\psi\leq 1$. Let us denote by $\lambda=\frac{1}{\beta l^{\beta+1}}$, then we observe that 
\begin{equation}\label{tstfn}
\lambda|F_l'(t)|\leq 1.
\end{equation}
We apply 
$$
\phi_1=\lambda F_l(u)\psi^p
\quad\text{and}\quad 
\phi_2=G_l(u)\psi^\frac{p}{2}
$$ 
as test functions in Definition \ref{Variationalmin}.
As in \cite{KinMar} with \eqref{tstfn}, see also \cite[Theorem 2.14 and Proposition 8.8] {AJBjorn}, we have
\begin{equation}\label{phi1}
g_{u+\phi_{1}}^p\leq(1-\lambda|F_l'(u)|\psi^p)g_u^{p}+\lambda p^pF_l(u)^p|F_l'(u)|^{1-p}g_{\psi}^p,
\end{equation}
and
\begin{equation}\label{phi2}
g_{\phi_2}\leq\frac{p}{2}\psi^{\frac{p}{2}-1}G_l(u)g_{\psi}+\psi^\frac{p}{2}G_l'(u)g_u.
\end{equation}
We divide the proof into the following steps.

Step 1. Applying $\phi=\phi_1$ as a test function in \eqref{weak} and using \eqref{phi1} we get
\begin{equation}\label{weakenergy}
\int_{X}|F_l'(u)|g_u^{p}\psi^p\,d\mu\leq p^p\int_{X}{F_l(u)^p}{|F_l'(u)|^{1-p}}g_{\psi}^p\,d\mu-\int_{X}f(u)F_l(u)\psi^p\,d\mu,
\end{equation}
where $f(u)=-u^{-\delta}$.
Using \eqref{cond3} into \eqref{weakenergy} we obtain
\begin{equation}\label{weakenergy1}
\int_{X}|F_l'(u)|g_u^{p}\psi^p\,d\mu\leq C(p,\beta)\int_{X}u^{p-\beta-1}g_{\psi}^p\,d\mu-\int_{X}f(u)F_l(u)\psi^p\,d\mu.
\end{equation}

Step 2. Choosing $\phi=\phi_2$ as a test function in \eqref{stable} and applying \eqref{phi2} for $f(u)=-u^{-\delta}$, we obtain
\begin{equation}\label{sta}
\int_{X}f'(u)G_l(u)^2\psi^p\,d\mu\leq I_1+I_2+I_3,
\end{equation}
where
$$
I_1=(p-1)\frac{p^2}{4}\int_{X}g_u^{p-2}\psi^{p-2}G_l(u)^2g_{\psi}^2\,d\mu,
$$
$$
I_2=p(p-1)\int_{X}\psi^{p-1}G_l(u)|G_l'(u)|g_u^{p-1}g_\psi\,d\mu
$$
and
$$
I_3=(p-1)\int_{X}g_u^{p}|G_l'(u)|^2\psi^p\,d\mu.
$$
We estimate the integrals $I_1$ and $I_2$ separately. 
When $p=2$, we have
\begin{equation}\label{estX2}
I_1=\int_{X}G_l(u)^2g_{\psi}^2\,d\mu,
\end{equation}
and when $p>2$, for every $\epsilon\in(0,1)$, we have
\begin{equation}\label{estXp}
\begin{split}
I_1&=(p-1)\frac{p^2}{4}\int_{X}\big(g_u^{p-2}|G_l'(u)|^{\frac{2(p-2)}{p}}\psi^{p-2}\big)\big(G_l(u)^2|G_l'(u)|^{\frac{2(2-p)}{p}}g_{\psi}^2\big)\,d\mu\\
&\leq\frac{\epsilon}{2}\int_{X}g_u^{p}|G_l'(u)|^2\psi^p\,d\mu+C(\epsilon,p)\int_{X}|G_l(u)|^p|G_l'(u)|^{2-p}g_{\psi}^p\,d\mu\\
&=\frac{\epsilon}{2(p-1)}I_3+C(\epsilon,p)\int_{X}|G_l(u)|^p|G_l'(u)|^{2-p}g_{\psi}^p\,d\mu,
\end{split}
\end{equation}
where we applied Young's inequality with exponents $\frac{p}{2}$ and $\frac{p-2}{2}$ for some positive constant $C(\epsilon,p)$. 
For every $p\geq 2$ and $\epsilon\in(0,1)$, using Young's inequality with exponents $p$ and $\frac{p}{p-1}$, we get  
\begin{equation}\label{estY}
\begin{split}
I_2&=p(p-1)\int_{X}\big(g^{p-1}_{u}\psi^{p-1}|G_l'(u)|^\frac{2(p-1)}{p}\big)\big(G_l(u)|G_l'(u)|^{\frac{2-p}{p}}g_{\psi}\big)\,d\mu\\
&\leq\frac{\epsilon}{2}\int_{X}g_u^{p}|G_l'(u)|^2\psi^p\,d\mu+C(\epsilon,p)\int_{X}G_l(u)^p|G_l'(u)|^{2-p}g_{\psi}^p\,d\mu\\
&=\frac{\epsilon}{2(p-1)}I_3+C(\epsilon,p)\int_{X}G_l(u)^p|G_l'(u)|^{2-p}g_{\psi}^p\,d\mu.
\end{split}
\end{equation}
Using \eqref{estX2}, \eqref{estXp} and \eqref{estY} in \eqref{sta} we arrive at
\begin{equation}\label{stableenergy}
\begin{split}
\int_{X}f'(u)G_l(u)^2\psi^p\,d\mu
&\leq (p-1+\epsilon)\int_{X}g_u^{p}|G_l'(u)|^2\psi^p\,d\mu\\
&\qquad+C(\epsilon,p)\int_{X}G_l(u)^p|G_l'(u)|^{2-p}g_{\psi}^p\,d\mu.
\end{split}
\end{equation}

Step 3. By applying \eqref{cond2} and \eqref{weakenergy1} in \eqref{stableenergy} we obtain
\begin{equation}\label{step3energy}
\begin{split}
&\int_{X}f'(u)G_l(u)^2\psi^p\,d\mu\leq\frac{(p-1+\epsilon)(\beta-1)^2}{4\beta}\int_{X}g^{p}_{u}|F_l'(u)|\psi^p\,d\mu\\
&\qquad+C(\epsilon,p)\int_{X}G_l(u)^p|G_l'(u)|^{2-p}g_{\psi}^p\,d\mu\\
&\leq\frac{(p-1+\epsilon)(\beta-1)^2}{4\beta}\Big(C(p,\beta)\int_{X}u^{p-\beta-1}g_{\psi}^p\,d\mu-\int_{X}f(u)F_l(u)\psi^p\,d\mu\Big)\\
&\qquad+C(\epsilon,p)\int_{X}G_l(u)^p|G_l'(u)|^{2-p}g_{\psi}^p\,d\mu\\
&\leq C(p,\beta,\epsilon)\int_{X}u^{p-\beta-1}g_{\psi}^p\,d\mu-\frac{(p-1+\epsilon)(\beta-1)^2}{4\beta}\int_{X}f(u)F_l(u)\psi^p\,d\mu,
\end{split}
\end{equation}
where the last inequality is obtained by \eqref{cond3}. By \eqref{cond1} and using the fact that $f(u)=-u^{-\delta}$ in \eqref{step3energy} we obtain
\begin{equation}\label{stabe4energy}
\beta_{\epsilon}\int_{X}u^{-\delta}F_l(u)\psi^p\,d\mu\leq C(p,\beta,\epsilon)\int_{X}u^{p-\beta-1}g_{\psi}^p\,d\mu,
\end{equation}
where
$$
\beta_{\epsilon}=\delta-\frac{(p-1+\epsilon)(\beta-1)^2}{4\beta}.
$$
Since $\beta\in J$, we have
$$
\lim_{\epsilon\to 0^{+}}\beta_{\epsilon}=\delta-\frac{(p-1)(\beta-1)^2}{4\beta}>0.
$$
Hence we can fix $\epsilon\in(0,1)$ which depends on $\delta,p,\beta$ such that $\beta_{\epsilon}>0$.
By the monotone convergence theorem, by letting $l\to\infty$ in \eqref{stabe4energy}, we have
\begin{equation}\label{sta5}
\int_{X}u^{-\delta-\beta}\psi^p\,d\mu\leq C(\delta,p,\beta)\int_{X}u^{p-\beta-1}g_{\psi}^p\,d\mu.
\end{equation}
Letting $l\to\infty$ in \eqref{weakenergy1} and using \eqref{sta5} we obtain
\begin{equation}\label{sta6}
\begin{split}
\int_{X}u^{-\beta-1}g_u^{p}\psi^p\,d\mu&\leq C(\delta,p,\beta)\int_{X}u^{p-\beta-1}g_{\psi}^p\,d\mu+\int_{X}u^{-\delta-\beta}\psi^p\,d\mu\\
&\leq C(\delta,\beta,p)\int_{X}u^{p-\beta-1}g_{\psi}^p\,d\mu.
\end{split}
\end{equation}
From \eqref{sta5} and \eqref{sta6} we get
\begin{equation*}\label{sta7}
\int_{X}(u^{-\delta-\beta}+g_u^{p}u^{-\beta-1})\psi^p\,d\mu\leq C(\delta,\beta,p)\int_{X}u^{p-\beta-1}g_{\psi}^p\,d\mu.
\end{equation*}
This proves the lemma.
\end{proof}

\section{Proof of Theorem \ref{mainthm}.}

Let $u\in N^{1,p}_{\mathrm{loc}}(X)$ be a positive variational minimizer corresponding to the problem \eqref{plap}. For $R>0$, let $\psi_R\in\text{Lip}_c(X)$ be such that 
\begin{align*}
&0\leq\psi_R\leq 1 \text{ in } X, \\
&\psi_R\equiv 1 \text{ in } B(0,R) \text{ and } \psi_R\equiv 0 \text { in } X\setminus B(0,2R),\\
&g_{\psi_{R}}\leq\frac{C}{R} \text{ for some positive constant } C \text{ independent of } R.
\end{align*}
We apply Lemma \ref{energy} with $\psi=\psi_R^\frac{\delta+\beta}{\delta+p-1}$ in \eqref{energyineq}. 
Since
$$
g_{\psi}\leq\Big(\frac{\delta+\beta}{\delta+p-1}\Big)\psi_R^\frac{\beta-p+1}{\delta+p-1}g_{\psi_R},
$$
for every $\beta\in(p-1,\delta_p)$, using Young's inequality with exponents $\frac{\delta+\beta}{\beta-p+1}$ and $\frac{\delta+\beta}{\delta+p-1}$, we obtain
\begin{equation*}\label{prefinalestimate}
\begin{split}
\int_{X}(u^{-\beta-\delta}+g_{u}^{p}u^{-\beta-1})\psi_{R}^\frac{p(\delta+\beta)}{\delta+p-1}\,d\mu&\leq
C\Big(\frac{\delta+\beta}{\delta+p-1}\Big)^{p}\int_{X}u^{p-\beta-1}\psi_{R}^{\frac{p(\beta-p+1)}{\delta+p-1}}g_{\psi_{R}}^p\,d\mu\\
&\leq\frac{1}{2}\int_{X}\Big(u^{p-\beta-1}\psi_{R}^{\frac{p(\beta-p+1)}{\delta+p-1}}\Big)^\frac{\delta+\beta}{\beta-p+1}\,d\mu+C(\delta,\beta,p)\int_{X}g_{\psi_{R}}^\frac{p(\delta+\beta)}{\delta+p-1}\,d\mu\\
&=\frac{1}{2}\int_{X}u^{-\beta-\delta}\psi_{R}^\frac{p(\delta+\beta)}{\delta+p-1}\,d\mu+C(\beta,\delta,p)\int_{X}g_{\psi_{R}}^\frac{p(\delta+\beta)}{\delta+p-1}\,d\mu,
\end{split}
\end{equation*}
which gives
\begin{equation}\label{prefinal}
\begin{split}
\int_{X}(\frac{1}{2}u^{-\beta-\delta}+g_{u}^{p}u^{-\beta-1})\psi_{R}^\frac{p(\delta+\beta)}{\delta+p-1}\,d\mu&\leq C(\beta,\delta,p)\int_{X}g_{\psi_{R}}^\frac{p(\delta+\beta)}{\delta+p-1}\,d\mu.
\end{split}
\end{equation}
Now using the properties of $\psi_R$, from \eqref{prefinal} we get
\begin{equation}\label{finalestimate}
\begin{split}
\int_{B(0,R)}(\frac{1}{2}u^{-\beta-\delta}+g_{u}^{p}u^{-\beta-1})\,d\mu&\leq C(\delta,\beta,p)\int_{B(0,2R)}R^{-\frac{p(\delta+\beta)}{\delta+p-1}}\,d\mu\\
&\leq C(\delta,\beta,p)C(C_\mu)R^{m-\frac{p(\delta+\beta)}{\delta+p-1}},
\end{split}
\end{equation}
where the last inequality is obtained by Lemma \ref{doublingimply} with $m=\log_2\,C_\mu$. 
By the given condition,
$$
\lim_{\beta\to\delta_p}\Big(m-\frac{p(\delta+\beta)}{\delta+p-1}\Big)<0.
$$
Thus there exists $\beta=\beta_{0}\in J$ such that 
$$
m-\frac{p(\delta+\beta_{0})}{\delta+p-1}<0
$$
By letting $R\to\infty$ in \eqref{finalestimate} we obtain
$$
\int_{X}(\frac{1}{2}u^{-\beta_{0}-\delta}+g_{u}^{p}u^{-\beta_{0}-1})\,d\mu=0,
$$ 
which is a contradiction. This completes the proof.

\bibliographystyle{plain}
\bibliography{mybibfile.bib}

\noindent
P.G.,
Department of Mathematics and System Analysis,
P.O. Box 11100, FI-00076,
Aalto University, Finland\\
Email: pgarain92@gmail.com

\medskip
\noindent
J.K.,
Department of Mathematics and System Analysis,
P.O. Box 11100, FI-00076
Aalto University, Finland\\
Email: juha.k.kinnunen@aalto.fi
\end{document}